\documentclass[12pt]{article}
\usepackage[margin=1in]{geometry}

\usepackage{url,subfig,mathrsfs,paralist}
\RequirePackage{amsthm,amsmath,amssymb}
\usepackage[colorlinks=true,
            linkcolor=red,
            citecolor=blue]{hyperref}
\usepackage{color}
\theoremstyle{plain}
\newtheorem{theorem}{Theorem}%[section]

\newtheorem{lemma}[theorem]{Lemma}

\theoremstyle{definition}

\newtheorem{remark}[theorem]{Remark}

\usepackage{enumitem}

\renewcommand{\leq}{\leqslant}

\renewcommand{\geq}{\geqslant}

\newcommand{\EE}{\mathrm{E}}

\newcommand{\RR}{\mathbb{R}}

\title{On stochastic heat equation with measure initial data %\footnote{}
 }
\author{Jingyu Huang %\footnote{ \texttt{}}
    }
\begin{document}
\maketitle
\begin{abstract}
	The aim of this short note is to obtain the existence, uniqueness and moment upper bounds of the solution to a stochastic heat equation with measure initial data, without using the iteration method in \cite{ChenDalang, ChenDalang2,  ChenKim}. \\


\end{abstract}
\section{Introduction}
Consider the stochastic heat equation
\begin{equation}\label{eq:SHE}
\frac{\partial u}{\partial t} =  \mathcal{L} u + b(u)+  \sigma(u)\dot{W}
\end{equation}
for $(t,x)\in (0,\infty)\times \RR^d$($d\geq 1$) where $\mathcal{L}$ is the generator of a L\'evy process $X=\{X_t\}_{t\geq 0}$.  
$\dot{W}$ is a centered Gaussian noise with covariance formally given by 
\begin{equation*}
\EE \big(\dot{W}(t,x) \dot{W}(s,y) \big) = \delta (s-t) f(x-y)\,,
\end{equation*}
where $f$ is some nonnegative and nonnegative definite function whose Fourier transform is denoted by 
\begin{equation*}
\hat{f}(\xi) := \int_{\RR^d} f(x)e^{-i x \xi}dx
\end{equation*}
in distributional sense, and $\delta$ denotes the Dirac delta function at $0$. For some technical reasons, we will assume that $f$ is lower semicoutinous (see Lemma \ref{eq:lem Davar} below). 

%Let $m_t$ denote the distribution of $X_t$ for every $t\geq 0$; that is, 
%\begin{equation*}
%m_t(A):= P \{X_t \in A\} \quad \text{for all} \  t\geq 0\ \text{and Borel sets}\ A \subset \RR^d\,,
%\end{equation*}
%and we will assume that the process $X$ has transition functions, i.e., 
%\begin{equation}\label{eq:transition fnc}
%$m_t(dx) \ll dx$ for all $t >0$.
%\end{equation}
%for some Borel measurable function $p_t(x):(0,\infty)\times \RR^d \to \RR_+$. 
Let $\Phi$ be the L\'evy exponent of $X_t$, we will assume that
\begin{equation}\label{eq:exp Phi in Lt}  
\exp (-\text{Re}\Phi) \in L^t(\RR^d) \ \text{for all}\  t>0\,.
\end{equation}
Thus according to Proposition 2.1 in \cite{FoondunKhoshnevisan}, $X_t$ has a transition function $p_t(x)$ and we can (and will) find a version of $p_t(x)$ which is continuous on $(0, \infty)\times \RR^d$ and uniformly continuous for all $(t,x)\in [\eta, \infty)\times \RR^d$ for every $\eta>0$, and that $p_t$ vanishes at infinity for all $t>0$. 

 The initial condition $u(0, \cdot)$ is assumed to be a (positive) measure $\mu(\cdot)$ such that
\begin{equation}\label{eq:initial cond int}
 \int_{\RR^d} p_t(x-y) \mu(dy) < \infty \quad \text{for all $t>0$ and $x\in \RR^d$} \,.
\end{equation}
To avoid trivialities, we assume that $\mu(\cdot) \not\equiv 0$.
%here $p_t(x)$ is the density of the random variable $X_t$. 

Using iteration method, the existence, uniqueness and some moment bounds of the solution have been obtained in \cite{ChenDalang, ChenDalang2,  ChenKim} for the case $b\equiv 0$ and for some specific choice of $\mathcal{L}$. However, these approaches rely on the structure (or asymptotic structure) of $p_t(x)$.  In this article, we will study the equation \eqref{eq:SHE} with also a Lipschitz drift term $b$ and establish the existence, uniqueness and $p$-th moment upper bound, without using the iteration method in  \cite{ChenDalang, ChenDalang2,  ChenKim}, also, our cribteria only need some integrability of the L\'evy exponent. 

To state the result, let us recall that by a solution $u$ to \eqref{eq:SHE} we mean a mild solution. That is, (i) $u$ is a predictable random field on a complete probability space $\{\Omega, \mathcal{F}, P \}$, with respect to the Brownian filtration generated by the cylindrical Brownian motion defined by $B_t(\phi):= \int_{[0,t]\times \RR^d} \phi(y) W(ds,dy)$, for all $t \geq 0$ and measurable $\phi : \RR^d \to \RR$ such that $\int_{\RR^d\times \RR^{d}} \phi(y)\phi(z)f(y-z)dydz < \infty$; and (ii) 
 %A random field $u=\{u(t,x), (t,x) \in (0, \infty) \times \RR^d\}$ is called a random field solution to \eqref{eq:SHE} if $u$ is predictable, %$\|u(t,x)\|_{L^2(\Omega)}< \infty$  for all $(t,x) \in (0, \infty)\times \RR^d$ 
for any $(t,x) \in (0, \infty) \times \RR^d$, the following equation holds a.s. 
\begin{align}\label{eq:mild}
u(t, x)= \int_{\RR^d}p_t(x-y) \mu(dy) &+ \int_0^t \int_{\RR^d} p_{t-s}(x-y) b(u(s,y)) dy ds \notag\\
 & + \int_0^t \int_{\RR^d} p_{t-s}(x-y) \sigma(u(s,y))W(ds,dy)\,.
\end{align}
where $p_t(x)$ is the transition function for $X_t$ and the stochastic integral above is in the sense of Walsh \cite{Walsh}. The following theorem is the main result of this paper. 
%\end{definition}

\begin{theorem}\label{thm:existence bd}
Assume that the initial condition satisfies \eqref{eq:initial cond int} and assume that 
\begin{equation}\label{eq:Dalang condition}
\Upsilon(\beta) := \sup_{t>0} \int_0^t \int_{\RR^d} \exp \left[ -2s \mathrm {Re}\Phi\left((1-\frac{s}{t})\xi\right) - 2(t-s)\mathrm{Re}\Phi\left(\frac{s}{t}\xi\right) \right] e^{-2\beta (t-s)} \hat{f}(\xi)d\xi ds < \infty
\end{equation}
and
\begin{equation}\label{eq: Dalang condition ori}
\tilde{\Upsilon}(\beta) := \int_{\RR^d} \frac{\hat{f}(\xi)d\xi}{\beta + \mathrm{Re}\Phi(\xi)} < \infty
\end{equation}
for any $\beta>0$. And assume that $\sigma$ and $b$ are Lipschitz functions with Lipschitz coefficients $L_{\sigma}, L_b >0$ respectively. Then there exists a unique mild solution to equation \eqref{eq:SHE}. Moreover, 
define 
\begin{equation}
\bar{\gamma}(p) := \limsup_{t \to \infty}\frac{1}{t} \sup_{x\in \RR^d}\log \left\|\frac{u(t,x)}{\tau + p_t*\mu(x)}\right\|_{L^p(\Omega)}\,,
\end{equation}
where  
\begin{equation}\label{eq:tau}
\tau = \max\left\{\frac{|b(0)|}{L_b}, \frac{|\sigma(0)|}{L_{\sigma}}\right\}\,.
\end{equation}
Then, 
\begin{equation}\label{eq:upper gamma}
\bar{\gamma}(p)\leq \inf\{\beta>0: B(\beta, p)<1\} \quad \text{for all integers $p \geq 2$}\,,
\end{equation}
where 
\begin{equation}\label{eq:B}
B(\beta, p) :=  \frac{L_b}{\beta}+    \frac{z_p L_{\sigma}}{(2\pi)^{d/2}}\left(\sqrt{\frac{\tilde{\Upsilon}(\beta)}{2} }+ \sqrt{\Upsilon(\beta)} \right) \,,
\end{equation}
and $z_p$ denotes the largest positive zero of the Hermite polynomial $He_p$. 
\end{theorem}
\begin{remark}
If we choose $\mathcal{L}$ to be the generator of an $a$-stable L\'evy process $D_{\theta}^a$ for $1< a < 2$, where $\theta$ is the skewness and $|\theta|< 2-a$(see \cite{ChenDalang2}),  or the Laplacian $\frac{1}{2}\Delta$ ($a=2$), then the classical Dalang's condition
\begin{equation}
\int_{\RR^d} \frac{\hat{f}(d\xi)}{1+|\xi|^a} < \infty
\end{equation}
implies condition \eqref{eq:Dalang condition}, since in this case $\text{Re}\Phi(\xi) = C|\xi|^a$ for some $C>0$. Also, in the case $d=1$ and $\dot{W}$ is a space-time white noise, that is, $f(\xi)\equiv 1$, condition \eqref{eq: Dalang condition ori} clearly guarantees that \eqref{eq:exp Phi in Lt} holds. 
\end{remark}

\begin{remark}
(Borrowed from \cite[Remark 1.5]{FoondunKhoshnevisan}). Recall that 
\begin{equation*}
He_k(x) = 2^{-k/2} H_k(x/\sqrt{2})  \quad \text{for all integers }\ k\geq 0 \ \text{and}\ x \in \RR\,,
\end{equation*}
where $\{H_k\}_{k=0}^{\infty}$ is defined uniquely via the following:
\begin{equation*}
e^{-2xt - t^2} = \sum_{k=0}^{\infty} \frac{t^k}{k!} H_k(x) \quad (t>0, x\in \RR)\,.
\end{equation*}
\end{remark}
%\section{preliminaries}
%The Gaussian random perturbation is described as follows: $W$ is a zero mean Gaussian family of random variables $\{W(\varphi), \varphi \in C_0^{\infty}([0,\infty)\times\RR^{d})\}$, where $C_0^{\infty}([0,\infty)\times\RR^{d})$ denotes the space of infinitely differentiable functions with compact support, defined in a complete probability space $\Omega, \mathcal{F}, P$, with covariance 
%\begin{equation}\label{eq:cov}
%\EE (W(\varphi)W(\psi)) = \int_0^{\infty} \int_{\RR^d} \int_{\RR^d} \varphi(t,x) f(x-y) \phi(t,y) dx dy dt\,.
%\end{equation}
%The covariance \eqref{eq:cov} can also be written using Fourier transform as 
%\begin{equation}
%\EE (W(\varphi)W(\psi)) = \frac{1}{(2\pi)^d}\int_0^{\infty} \int_{\RR^d} \mathcal{F}\varphi(t)(\xi) \overline{\mathcal{F}\psi(t)(\xi)} \mu(d\xi) dt\,.
%\end{equation}
%The completion of the Schwartz space $\mathcal{S}(\RR^d)$ of rapidly decreasing $C^{\infty}$ functions, endowed with the inner product
%\begin{equation}
%\langle \varphi, \psi\rangle_{\mathcal{H}} = \int_{\RR^d} \int_{\RR^d} \varphi(x) f(x-y) \psi(y) dx dy = \frac{1}{(2\pi)^d}\int_{\RR^d} \mathcal{F}\varphi(\xi) \overline {\mathcal{F}\psi (\xi)} \mu(d\xi)\,,
%\end{equation}
%$\varphi, \psi \in \mathcal{S}(\RR^d)$, is denoted by $\mathcal{H}$. Notice that $\mathcal{H}$ may contain distributions. Set $\mathcal{H}_T = L^2 ([0,T]; \mathcal{H})$. 

\section{Proof of Theorem \ref{thm:existence bd}}
%\begin{definition}
%A kernel [on $\RR^d$] is a Borel measurable function $\kappa: \RR^d \to [0, \infty]$ such that $\kappa \in L^1_{loc}(\RR^d)$. We say that $\kappa$ is a kernel of positive type if $\kappa$ is a kernel whose Foureir transform $\hat{\kappa}$ is a function that is nonnegative almost everywhere. 
%\end{definition}
In the proof of Theorem \ref{thm:existence bd} we will need two results about taking Fourier transforms, which we now state. 
\begin{lemma}[Corollary 3.4 in \cite{FoondunKhoshnevisan}]\label{eq:lem Davar}
Assume that $f$ is lower semicontinuous, then for all Borel probability measures $\nu$ on $\RR^d$, 
\begin{equation*}
\int_{\RR^d}\int_{\RR^d} f(x-y) \nu(dx)\nu(dy) = \frac{1}{(2\pi)^d} \int_{\RR^d} \hat{f}(\xi) |\hat{\nu}(\xi)|^2 d\xi\,.
\end{equation*}
\end{lemma}

\begin{lemma}\label{lem:bridge fourier}
If $f$ is lower semicontinuous, then 
\begin{align*}
&\int_{\RR^d} \int_{\RR^d} p_{t-s}(x-y_1) p_s*\mu(y_1)p_{t-s}(x-y_2)p_s*\mu(y_2)f(y_1-y_2) dy_1dy_2\\
 &\qquad\leq \frac{ [ p_t*\mu(x)]^2}{(2\pi)^d} \int_{\RR^d} \exp \left[ -2s \mathrm{Re} \Phi\left((1-\frac{s}{t})\xi\right)- 2(t-s)\mathrm{Re}\Phi\left(\frac{s}{t}\xi\right)\right] \hat{f}(\xi)d\xi\,.
\end{align*}
\end{lemma}
\begin{proof}
We begin by noting that 
\begin{align*}
&\int_{\RR^d}\int_{\RR^d}p_{t-s}(x-y_1) p_s*\mu(y_1)p_{t-s}(x-y_2) p_s*\mu(y_2)f(y_1-y_2)dy_1 dy_2\\
&\qquad = \int_{\RR^d}\int_{\RR^d}\int_{\RR^d}\int_{\RR^d} \frac{p_{t-s}(x-y_1)p_s(y_1-z_1)}{p_t(x-z_1)}\frac{p_{t-s}(x-y_2)p_s(y_2-z_2)}{p_t(x-z_2)}f(y_1-y_2)dy_1dy_2\\
&\qquad\qquad\qquad \times p_t(x-z_1)p_t(x-z_2)\mu(dz_1)\mu(dz_2)\,,
\end{align*}
and as a function of $y$, the quotient 
%\begin{equation}
$\frac{p_{t-s}(x-y)p_s(y-z)}{p_t(x-z)}$
%\end{equation}
is the probability density of the L\'evy bridge $\tilde{X}_{z, x, t}=\{\tilde{X}_{z, x, t}(s)\}_{0 \leq s \leq t}$ which is at $z$ when $s=0$ and at $x$ when $s=t$. Actually, $\tilde{X}_{z, x, t}(s)$ can be written as 
\begin{align*}
\tilde{X}_{z, x, t}(s) &= X_s-\frac{s}{t}X_t + z + \frac{s}{t}(x-z)\\
 &= (1-\frac{s}{t}) X_s - \frac{s}{t} (X_t-X_s) + z + \frac{s}{t}(x-z)\,,
\end{align*}
hence by the independence of increment of L\'evy process, we have
\begin{equation*}
\EE e^{i \xi \tilde{X}_{z, x, t}(s)} = \exp \left(-s \Phi\left((1-\frac{s}{t})\xi\right) - (t-s) \Phi\left(-\frac{s}{t}\xi\right) \right) e^{i (z+ \frac{s}{t}(x-z))}\,.
\end{equation*}
Thus, an application of Lemma \ref{eq:lem Davar} to $\nu_j(dy) = \frac{p_{t-s}(x-y)p_s(y-z_j)}{p_t(x-z_j)}dy$, $j=1,2$, yields
\begin{align*}
&\int_{\RR^d} \int_{\RR^d} \frac{p_{t-s}(x-y_1)p_s(y_1-z_1)}{p_t(x-z_1)} \frac{p_{t-s}(x-y_2)p_s(y_2-z_2)}{p_t(x-z_2)} f(y_1-y_2) dy_1 dy_2 \\
&\qquad= \frac{1}{(2\pi)^d} \int_{\RR^d} \EE e^{i \xi \tilde{X}_{z_1, x, t}(s)} \overline{\EE e^{i \xi \tilde{X}_{z_2, x, t}(s)}} \hat{f}(\xi)d\xi \\
&\qquad\leq  \frac{1}{(2\pi)^d} \int_{\RR^d}\exp \left[-2s \text{Re}\Phi\left((1-\frac{s}{t})\xi\right) - 2(t-s) \text{Re}\Phi\left(-\frac{s}{t}\xi\right) \right] \hat{f}(\xi)d\xi\,,
\end{align*}
which proves the lemma. 
\end{proof}

To prove Theorem \ref{thm:existence bd}, we first define a norm for all $\beta, p >0$ and all predictable random fields $v:=v(t,x)$, 
\begin{equation}
\|v\|_{\beta, p} = \sup_{t > 0} e^{-\beta t} \sup_{x \in \RR^d} \|v(t,x)\|_{L^p(\Omega)}\,.
\end{equation}
Let $\mathcal{B}_{\beta, p}$ denote the collection of all predictable random fields $v:= \{v(t,x)\}_{t\geq 0, x\in \RR^d}$ such that $\|v\|_{\beta, p} < \infty$. We note that after the usual identification of a process with its modifications, $\mathcal{B}_{\beta, p}$ is a Banach space (see Section 5 in \cite{FoondunKhoshnevisan}).

\begin{proof}[Proof of Theorem \ref{thm:existence bd}]
%\begin{equation}\label{eq:divide}
%\frac{\tau + |u(t,x)|}{\tau + p_t*\mu(x)} \leq 1 + \left| \int_0^t \int_{\RR^d} \frac{p_{t-s}(x-y) (\tau+ p_s*\mu(y))}{\tau + p_t*\mu(x)} \frac{\sigma(u(s,y))}{\tau+p_s*\mu(y)} W(ds,dy)\right|\,,
%\end{equation}
%we will study the random field 
%\begin{equation}
%\frac{\tau+ u(t,x)}{\tau + p_t*\mu(x)}\,.
%\end{equation}

We use Picard iteration. Set 
\begin{align*}
&u^0(t,x):= p_t*\mu(x)\,,\\
& u^{n+1}(t,x):= p_t*\mu(x)+ \int_0^t \int_{\RR^d} p_{t-s}(x-y)b(u^n(s,y))dy ds \\
&\qquad \qquad\qquad\qquad\qquad+ \int_0^t \int_{\RR^d} p_{t-s}(x-y) \sigma(u^n(s,y))W(ds,dy)\,.
\end{align*}
We first show that whenever $\beta$ is chosen such that $B(\beta, p)<1$, where $B(\beta, p)$ is defined in \eqref{eq:B},  then,  for any $n \geq 1$, 
\begin{equation}\label{eq:uniform bd in t and x}
\left\|\frac{\tau+ |u^{n}|}{\tau+p*\mu}\right\|_{\beta, p} < \infty\,.
\end{equation}
Note that by the dominated convergence theorem, the condition $B(\beta, p)<1$ can be achieved if $\beta$ is sufficiently large. 

Recall that $\tau$ is defined in \eqref{eq:tau}. We start with the inequality 
\begin{align*}
\frac{\tau + |u^{n+1}(t,x)|}{\tau + p_t*\mu(x)} \leq &1 + \left| \int_0^t \int_{\RR^d} \frac{p_{t-s}(x-y)[\tau + p_s*\mu(y)]}{\tau + p_t* \mu(x)} \frac{b(u^n(s,y))}{\tau + p_s*\mu(y)}dy ds\right|\\
&+ \left|\int_0^t \int_{\RR^d} \frac{p_{t-s}(x-y) (\tau + p_s*\mu(y))}{\tau + p_t*\mu(x)} \frac{\sigma(u^n(s,y))}{\tau + p_s*\mu(y)} W(ds,dy)\right|\,.
\end{align*}
\eqref{eq:uniform bd in t and x} is clearly true for $n=0$. Using induction, assume \eqref{eq:uniform bd in t and x} is true for some $n$, using Burkholder inequality (see \cite{Davis}) and the assumption on $\sigma$ and $b$ , we obtain
\begin{align*}
&\left\|\frac{\tau+ |u^{n+1}(t,x)|}{\tau+p_t*\mu(x)}\right\|_{L^p(\Omega)}\\
%\leq & 1+ \int_0^t \int_{\RR^d} \frac{p_{t-s}(x-y)[\tau+p_s*\mu(y)]}{\tau + p_t*\mu(x)} \left\| \frac{b(u^n(s,y))}{\tau + p_s*\mu(y)}\right\|_{L^p(\Omega)}dy ds\\
%&+z_p \bigg( \int_0^t \int_{\RR^d} \int_{\RR^d} \frac{p_{t-s}(x-y_1) (\tau+ p_s* \mu(y_1))}{\tau+p_t*\mu(x)}\frac{p_{t-s}(x-y_1) (\tau+ p_s* \mu(y_1))}{\tau+p_t*\mu(x)} \\
%&\qquad\qquad\quad\times\left\| \frac{\sigma(u^n(s,y_1))}{\tau+ p_s*\mu(y_1)}\frac{\sigma(u^n(s,y_2))}{\tau+p_s*\mu(y_2)}\right\|_{L^{p/2}(\Omega)} f(y_1-y_2) dy_1 dy_2 ds\bigg)^{1/2}\\
\leq& 1+ L_b \int_0^t \int_{\RR^d} \frac{p_{t-s}(x-y)[\tau + p_s*\mu(y)]}{\tau+p_t*\mu(x)} \left\|\frac{\tau + |u^n(s,y)|}{\tau + p_s*\mu(y)}\right\|_{L^p(\Omega)}dy ds\\
&+z_p L_{\sigma} \bigg( \int_0^t \int_{\RR^d}\int_{\RR^d}  \frac{p_{t-s}(x-y_1) (\tau+ p_s* \mu(y_1))}{\tau+p_t*\mu(x)}\frac{p_{t-s}(x-y_1) (\tau+ p_s* \mu(y_1))}{\tau+p_t*\mu(x)} \\
&\qquad \qquad\times \left\| \frac{\tau + |u^n(s,y_1)|}{\tau + p_s* \mu(y_1)}\right\|_{L^p(\Omega)} \left\| \frac{\tau + |u^n(s,y_2)|}{\tau + p_s* \mu(y_2)}\right\|_{L^p(\Omega)} f(y_1-y_2) dy_1 dy_2 ds\bigg)^{1/2}\,,
\end{align*}
multiplying both sides by $e^{-\beta t}$ and applying Minkowski's inequality to the third summand above we obtain
\begin{align*}
&\qquad e^{-\beta t} \left\| \frac{\tau + |u^{n+1}(t,x)|}{\tau + p_t* \mu(x)}\right\|_{L^p(\Omega)}\\
& \leq 1 
 + L_b\left\| \frac{\tau + |u^n|}{\tau + p* \mu} \right\|_{\beta, p}\int_0^t \int_{\RR^d} e^{-\beta(t-s)} \frac{p_{t-s}(x-y)[\tau + p_s*\mu(y)]}{\tau + p_t* \mu(x)} dy ds \\
 &\quad +  z_p L_{\sigma}\left\|\frac{\tau + |u^n|}{\tau + p*\mu}\right\|_{\beta, p} \left( \int_0^t \int_{\RR^d} \int_{\RR^d} e^{-2\beta(t-s)} p_{t-s}(x-y_1) p_{t-s}(x-y_2) f(y_1-y_2) dy_1 dy_2 ds\right)^{1/2}   \\
 & \quad + z_p L_{\sigma}  \left\|\frac{\tau + |u^n|}{\tau + p*\mu}\right\|_{\beta, p} \\
 &\qquad \times\left( \int_0^t \int_{\RR^d} \int_{\RR^d} e^{-2\beta(t-s)}\frac{p_{t-s}(x-y_1)p_s*\mu(y_1)}{p_t*\mu(x)} \frac{p_{t-s}(x-y_2)p_s*\mu(y_2)}{p_t*\mu(x)} f(y_1-y_2)dy_1 dy_2 ds\right)^{1/2} \\
 &:= 1+ I_1 + I _2 + I _3\,,
\end{align*}
where in obtaining $I_2$ and $I_3$ above, we have used the bound
\begin{equation}\label{eq: bd tau heat kernel}
\frac{p_{t-s}(x-y) \tau}{\tau + p_t*\mu(x)} \leq p_{t-s}(x-y) \quad \text{and} \quad \frac{p_{t-s}(x-y)p_s*\mu(y)}{\tau + p_t*\mu(x)} \leq \frac{p_{t-s}(x-y)p_s*\mu(y)}{ p_t*\mu(x)}\,.
\end{equation}
%then using Minkowski's inequality we have
%\begin{align*}
%&\bigg( \int_0^t \int_{\RR^d}\int_{\RR^d}  \frac{p_{t-s}(x-y_1) (\tau+ p_s* \mu(y_1))}{\tau+p_t*\mu(x)} \frac{p_{t-s}(x-y_1) (\tau+ p_s* \mu(y_1))}{\tau+p_t*\mu(x)} f(y_1-y_2) dy_1 dy_2 ds\bigg)^{1/2}\\ 
%\leq&\bigg( \int_0^t \int_{\RR^d}\int_{\RR^d}  p_{t-s}(x-y_1) p_{t-s}(x-y_1) f(y_1-y_2) dy_1 dy_2 ds\bigg)^{1/2}\\
%&+ \bigg( \int_0^t \int_{\RR^d}\int_{\RR^d}  \frac{p_{t-s}(x-y_1) (p_s* \mu(y_1))}{p_t*\mu(x)}
 %\frac{p_{t-s}(x-y_1) (p_s* \mu(y_1))}{p_t*\mu(x)} f(y_1-y_2) dy_1 dy_2 ds\bigg)^{1/2}\\
 %\leq& \left( \int_0^t \int_{\RR^d} e^{-2(t-s)|\xi|^2} \hat{f}(d\xi) ds \right)^{1/2}+ \left( \int_0^t \int_{\RR^d} e^{-\frac{2s(t-s)}{t}|\xi|^2}  \hat{f}(d\xi)ds \right)^{1/2} < \infty\,,
%\end{align*} 
%where the second summand comes from the relation of heat kernel 
We will estimate $I_1, I_2, I_3$ separately. For $I_1$, %we  will apply a property of heat kernel 
 %\begin{equation}\label{eq:property of heat kernel}
%p_{t-s}(x-y)p_s(y-z) = p_t(x-z) p_{\frac{s(t-s)}{t}} \left(y-z-\frac{s}{t} (x-z)\right)\,,
%\end{equation} 
the semigroup property of $p_t(x)$ yields  
%\begin{align*}
%&\int_0^t \int_{\RR^d} e^{-\beta(t-s)} \frac{p_{t-s}(x-y)[\tau + p_s*\mu(y)]}{\tau + p_t* \mu(x)} dy ds\\
%\leq &\int_0^t \int_{\RR^d} e^{-\beta(t-s)} p_{t-s}(x-y)dy ds+ \int_0^t \int_{\RR^d} e^{-\beta(t-s)} \frac{p_{t-s}(x-y) p_s*\mu(y)}{p_t*\mu(x)}dy ds\\
%\leq& \int_0^t e^{-\beta(t-s)} ds + \int_0^t \int_{\RR^d}  e^{-\beta (t-s)} \frac{1}{p_t* \mu(x)} \int_{\RR^d} p_t(x-z) p_{\frac{s(t-s)}{t}}(y-z-\frac{s}{t}(x-z)) \mu(dz) dy ds\\
 %=& \int_0^t e^{-\beta(t-s)} ds + \int_0^t \int_{\RR^d}  e^{-\beta (t-s)} \frac{1}{p_t* \mu(x)} p_t(x-z) \mu(dz) ds\\
% = &  2\int_0^t e^{-\beta(t-s)} ds \leq \frac{2}{\beta}\,,
%\end{align*}
%where in the third line we used the semigroup property of $p_t(x)$. Thus we obtain 
\begin{equation*}
I_1 \leq \frac{L_b}{\beta} \left\|\frac{\tau+ |u^n|}{\tau + p*\mu}\right\|_{\beta, p}\,.
\end{equation*}
For $I_2$, an application of Lemma \ref{eq:lem Davar} to $\nu(dy)=p_{t-s}(x-y)dy$ yields
\begin{align*}
&\int_0^t \int_{\RR^d} \int_{\RR^d} e^{-2\beta(t-s)} p_{t-s}(x-y_1) p_{t-s}(x-y_2) f(y_1-y_2) dy_1 dy_2 ds\\
&\qquad= \frac{1}{(2\pi)^{d}}\int_0^t \int_{\RR^d} e^{-2(t-s)\text{Re}\Phi(\xi)} \hat{f}(\xi)d\xi e^{-2\beta(t-s)} ds  \leq\frac{1}{2(2\pi)^d} \int_{\RR^d} \frac{\hat{f}(\xi)d\xi}{\beta + \text{Re}\Phi(\xi)}\,,
\end{align*}
thus we obtain
\begin{equation*}
I _2 \leq z_p L_{\sigma} \left( \frac{1}{2(2\pi)^d}\tilde{\Upsilon}(\beta)\right)^{1/2} \left\|\frac{\tau+ |u^n|}{\tau + p*\mu}\right\|_{\beta, p}\,.
\end{equation*}
Finally, an application of Lemma \ref{lem:bridge fourier} yields
%we obtain
%\begin{equation}
%p_{t-s}(x-y) p_s*\mu(y) = \int_{\RR^d} p_t(x-z) p_{\frac{s(t-s)}{t}}\left(y-z-\frac{s}{t}(x-z)\right) \mu(dz)\,.
%\end{equation}
%and with the Fourier transform yielding
%\begin{align}\label{eq: Fourier delete}
%&\int_{\RR^d}\int_{\RR^d} p_{\frac{s(t-s)}{t}}\left(y_1-z-\frac{s}{t}(x-z)\right)p_{\frac{s(t-s)}{t}}\left(y_2-z-\frac{s}{t}(x-z)\right)f(y_1-y_2) dy_1 dy_2 \\
%\leq& \int_{\RR^d}\int_{\RR^d}p_{\frac{s(t-s)}{t}}(y_1) p_{\frac{s(t-s)}{t}}(y_2)f(y_1-y_2) dy_1 dy_2 \notag \,,
%\end{align}
%thus 
%\begin{align*}
%&\int_{\RR^d}\int_{\RR^d}  \frac{p_{t-s}(x-y_1) p_s* \mu(y_1)}{p_t*\mu(x)}
 %\frac{p_{t-s}(x-y_1) p_s* \mu(y_1)}{p_t*\mu(x)} f(y_1-y_2) dy_1 dy_2\\
 %\leq &\int_{\RR^d} \exp \left[ -2s \text{Re} \Phi\left((1-\frac{s}{t})\xi\right)- 2(t-s)\text{Re}\Phi\left(\frac{s}{t}\xi\right)\right] \hat{f}(d\xi)
% \leq& \int_{\RR^d}\int_{\RR^d}\int_{\RR^d}\int_{\RR^d} \left(\frac{1}{p_t*\mu(x)}\right)^2 p_t(x-z_1) p_{\frac{s(t-s)}{t}}(y_1-z_1-\frac{s}{t}(x-z_1))\\
 %&\times p_t(x-z_2) p_{\frac{s(t-s)}{t}}(y_2-z_2-\frac{s}{t}(x-z_2)) f(y_1-y_2) dy_1 dy_2 \mu(dz_1)\mu(dz_2)\\
 %\leq&  \int_{\RR^d}\int_{\RR^d}\int_{\RR^d}\int_{\RR^d}\left(\frac{1}{p_t*\mu(x)}\right)^2 p_t(x-z_1) p_{\frac{s(t-s)}{t}}\left(y_1\right)p_t(x-z_2) p_{\frac{s(t-s)}{t}}\left(y_2\right)\\
% & \times  f(y_1-y_2) dy_1 dy_2 \mu(dz_1)\mu(dz_2)\\
 %=& \int_{\RR^d}\int_{\RR^d}  p_{\frac{s(t-s)}{t}}\left(y_1\right) p_{\frac{s(t-s)}{t}}\left(y_2\right)f(y_1-y_2) dy_1 dy_2 \,.
% \end{align*}
% Thus we obtain
 \begin{align*}
 I_3 \leq z_p L_{\sigma}  \left\|\frac{\tau+ |u^n|}{\tau + p*\mu}\right\|_{\beta, p}\left( \frac{1}{(2\pi)^d} \Upsilon(\beta) \right)^{1/2}\,.
 %&\leq z_p L_{\sigma} \left\|\frac{\tau+ |u^n|}{\tau + p*\mu}\right\|_{\beta, p} \left(\frac{1}{(2\pi)^d}\int_0^t \int_{\RR^d} e^{-\frac{s(t-s)}{t} |\xi|^2} e^{-2\beta(t-s)} \hat{f}(d\xi) ds\right)^{1/2}\\
% &\leq2  z_p L_{\sigma} \left( \frac{1}{(2\pi)^d}\int_{\RR^d} \frac{\hat{f}(d\xi)}{4\beta + |\xi|^2}\right)^{1/2} \left\|\frac{\tau+ |u^n|}{\tau + p*\mu}\right\|_{\beta, p}\,,
\end{align*}
%where the last inequality is from Lemma A.1 in \cite{ChenHuang}. 
Combining the estimates for $I_1, I_2, I_3$, we arrive at
\begin{equation*}
\left\|\frac{\tau+ |u^{n+1}|}{\tau + p*\mu}\right\|_{\beta, p} \leq 1 + B(\beta, p) \left\|\frac{\tau+ |u^{n}|}{\tau + p*\mu}\right\|_{\beta, p} \,,
\end{equation*}
where $B(\beta, p)$ is defined in \eqref{eq:B}. Using the iteration, we see that \eqref{eq:uniform bd in t and x} holds for all $n\geq 1$ if $B(\beta, p)<1$. 
%\begin{equation*}
%\sup_{n \geq 1} \left\|\frac{\tau+ |u^{n}|}{\tau + p*\mu}\right\|_{\beta, p} < \infty\,.
%\end{equation*}

%so \eqref{eq:uniform bd in t and x} is proved. A similar calculation shows that 
%\begin{align*}
%&\sup_{x \in \RR^d}\left\|\frac{\tau+ |u^{n+1}(t,x)|}{\tau+p_t*\mu(x)}\right\|_{L^p(\Omega)}\\
 %\leq & 1+ C_p  L \bigg( \int_0^t \sup_{z \in \RR^d} \left\|  \frac{\tau+ |u^n(s,z)|}{\tau+p_s*\mu(z)}\right\|^2_{L^p(\Omega)}\left( \int_{\RR^d} e^{-2(t-s)|\xi|^2} \hat{f}(d\xi) +  \int_{\RR^d} e^{-\frac{2s(t-s)}{t}|\xi|^2}  \hat{f}(d\xi) \right) ds \bigg)^{1/2}
%\end{align*}
%then an application of a version of Gronwall’s Lemma presented in Lemma 15 of \cite{Dalang} shows that the finiteness in \eqref{eq:uniform bd in t and x} is actually uniform in $n$.  

The same technique applied to $\frac{u^{n+1}(t,x) - u^n(t,x)}{\tau + p_t*\mu(x)}$
yields that 
\begin{align*}
\left\| \frac{u^{n+1} - u^n}{\tau + p*\mu}\right\|_{\beta, p} \leq B(\beta, p)  \left\| \frac{u^{n} - u^{n-1}}{\tau + p*\mu}\right\|_{\beta, p}\,,
%&\leq \left( \int_0^t \sup_{z \in \RR^d}\left\| \frac{u^{n}(s,z) - u^{n-1}(s,z)}{\tau + p_s*\mu(z)}\right\|_{L^p(\Omega)} \left( \int_{\RR^d} e^{-2(t-s)|\xi|^2} \hat{f}(d\xi) +  \int_{\RR^d} e^{-\frac{2s(t-s)}{t}|\xi|^2}  \hat{f}(d\xi) \right) ds \right)^{1/2}\,,
\end{align*}
and if $\beta$ is chosen such that $B(\beta, p)<1$, we will obtain that
\begin{equation*}
\sum_{n=1}^{\infty}\left\| \frac{u^{n} - u^{n-1}}{\tau + p*\mu}\right\|_{\beta, p} < \infty\,.
\end{equation*}
Therefore, we can find a predictable random field $u^{\infty} \in \mathcal{B}_{\beta,p}$ such that $\lim_{n \to \infty} u^n = u^{\infty}$ in $\mathcal{B}_{\beta, p}$. It is easy to see that this $u^{\infty}$ is a solution to equation \eqref{eq:mild}, and uniqueness is checked by a standard argument. 

To prove \eqref{eq:upper gamma},  we note that since $u\in \mathcal{B}_{\beta, p}$ for those $\beta$ such that $B(\beta, p)<1$, 
\begin{equation*}
\sup_{x\in \RR^d} \left\| \frac{u(t,x)}{\tau + p_t*\mu(x)}\right\|_{L^p(\Omega)} \leq \sup_{x\in \RR^d} \frac{\tau}{\tau+ p_t*\mu(x)} + C e^{\beta t} 
\end{equation*}
%for those $\beta $ such that $B(\beta, p) < 1$ and 
for some $C>0$ which does not depend on $t$, thus \eqref{eq:upper gamma} is proved and the proof of Theorem \ref{thm:existence bd} is complete.
%then another application of Lemma 15 of \cite{Dalang} shows that $\frac{u^n(t,x)}{\tau + p_t*\mu(x)} $ converges in $L^p(\Omega)$ to a limit $\frac{u(t,x)}{\tau + p_t*\mu(x)} $, and $u(t,x)$ solves equation \eqref{eq:mild}. Uniqueness follows from a standard argument. 
\end{proof}
%\begin{remark}
%The proof also shows that we have the convergence 
%\begin{equation}
%\sup_{t\leq T} \sup_{x \in \RR^d} \left\| \frac{u^n(t,x) - u(t,x)}{\tau + p_t*\mu(x)}\right\|_{L^p(\Omega)} \to 0
%\end{equation}
%as $n \to \infty$ for every $T>0$.
%\end{remark}

\section{Acknowledgement}
The author thanks Davar Khoshnevisan and David Nualart for stimulating discussions and encouragement.

\begin{small}

\end{small}

\begin{minipage}{0.6\textwidth}
\noindent\textbf{Jingyu Huang}\\
\noindent Department of Mathematics\\
\noindent University of Utah\\
\noindent Salt Lake City, UT 84112-0090\\
\noindent\emph{Email:} \texttt{jhuang@math.utah.edu}\\
\noindent\emph{URL:} \url{http://www.math.utah.edu/~jhuang/}\\
\end{minipage}

\end{document}